\documentclass[a4paper,10pt
]{article}

\usepackage{amsfonts}
\usepackage{amssymb}
\usepackage{amsthm}
\usepackage{amsmath, amsfonts, amsxtra}
\usepackage{verbatim} %
\usepackage{hyperref}%

\newtheorem{theorem}{Theorem}[section]%
\newtheorem{lemma}[theorem]{Lemma}%
\newtheorem{remark}[theorem]{Remark}%
\usepackage{array} 
\newenvironment{ma}{\begin{array}{>{\displaystyle}r >{\displaystyle}c >{\displaystyle}l}}{\end{array}}%

\newcommand{\N}{{\mathbb N}}
\newcommand{\R}{{\mathbb R}}

\newcommand{\lap}{\triangle}
\newcommand{\abs}[1]{{\left \vert #1 \right \vert}}%

\newcommand{\ontop}[2]{{\genfrac{}{}{0pt}{}{#1}{#2}}}

\newcommand{\intl}{\int \limits}

\def\XXint#1#2#3{{\setbox0=\hbox{$#1{#2#3}{\int}$}%
     \vcenter{\hbox{$#2#3$}}\kern-.5\wd0}}%
\numberwithin{equation}{section}%

\title{A Remark on Gauge Transformations and the Moving Frame Method}
\author{Armin Schikorra\footnote{supported by Studienstiftung des Deutschen Volkes Grant}}
\date{}

\begin{document}
\maketitle
\thispagestyle{empty}
\begin{abstract}
\noindent In this note we give a shorter proof of recent regularity results in \cite{Riviere06}, \cite{StruweRiviere}. We differ from the mentioned articles only in using the direct method of H\'{e}lein's moving frame to construct a suitable gauge transformation. Though this is neither new nor surprising, it enables us to describe a proof of regularity using besides the duality of Hardy- and BMO-space only elementary arguments of calculus of variations and algebraic identities. Moreover, we remark that in order to prove Hildebrandt's conjecture one can avoid the Nash-Moser imbedding theorem.\\
There are no new results presented here, nor are there any techniques we could claim originality for.
\\[1ex]
\end{abstract}
%

\section{Introduction}
In the influential article \cite{Riviere06} Rivi\`{e}re discovered that Euler equations of conformally invariant variational functionals acting on maps $U \in W^{1,2}(\mathcal{M},\mathcal{N})$ from two-dimensional manifolds $\mathcal{M}$ into $n$-dimensional manifolds $\mathcal{N}$ can locally be written in the form
\begin{equation}
 \label{eq:upde} \lap u^i = \Omega_{ik} \cdot \nabla u^k \quad \mbox{in }B_1(0),\quad 1 \leq i \leq n,
\end{equation}
where $\Omega_{ij} = - \Omega_{ji} \in L^2(B_1(0),\R^2)$ and $u \in W^{1´,2}(B_1(0),\mathcal{N})$ is a local representation of $U$. Here and in the following we adopt Einstein's summation convention, summing over repeated indices. For an overview of the geometric problems and the development towards the regularity result finally achieved, the interested reader is referred to the detailed introduction in \cite{Riviere06}.\\
The right hand side of \eqref{eq:upde} is only in $L^1$, and hence there is no standard theory in order to conclude better regularity as e.g. continuity of $u$. Using an algebraic feature, namely the antisymmetry of $\Omega$, one can construct a gauge transformation $P \in W^{1,2}(B_1(0),SO(n))$ which pointwise almost everywhere is an orthogonal matrix in $\R^{n \times n}$ such that
\begin{equation}\label{eq:divomegap1st}
\operatorname{div} (P^T_{ik} \nabla P_{kj} - P^T_{ik} \Omega_{kl} P_{lj}) = 0 \quad \mbox{in }B_1(0), \quad 1 \leq i,j \leq n.
\end{equation}
Statements on matrices like the last one will often be abbreviated by omitting matrix indices. That is, instead of \eqref{eq:divomegap1st} we will write
\begin{equation}\label{eq:divomegap}
 \operatorname{div} (P^T \nabla P - P^T \Omega P) = 0 \quad \mbox{in }B_1(0).
\end{equation}
Then, by solving an extra system of PDEs Rivi\`{e}re finds an invertible matrix $A \in W^{1,2}\cap L^\infty(B_1(0),GL(n))$ such that
\begin{equation}\label{eq:divomegaa}
 \operatorname{div} (\nabla A - A\Omega) = 0 \quad \mbox{in }B_1(0).
\end{equation}
Using this, \eqref{eq:upde} transforms into
\[
 \operatorname{div} (A\nabla u) = (\nabla A - A \Omega)\cdot \nabla u \quad \mbox{in }B_1(0).
\]
By \cite{Mueller90}, \cite{CLMS} the right hand side lies in the Hardy-space $\mathcal{H}$. This is a strict subspace of $L^1$ featuring a good behavior when being convoluted with Calderon-Zygmund kernels, implying continuity of $u$. (A great source on this is e.g. \cite{Stein93}, for an overview with a focus on PDE one might also want to look into \cite{Semmes94}). The way of constructing $A$ seems to be purely two-dimensional, as it crucially relies on $L^\infty$-bounds of Wente's inequality (for the statement see \cite[Lemma A.1]{Riviere06}, for proofs see \cite{Wente69}, \cite[Chapter II]{Tartar84}, \cite[Lemma A.1]{BrC84} or \cite[Chapter 3]{Helein}).\\
Adapting this idea in its spirit to higher dimensions, in \cite{StruweRiviere} it is shown how to prove regularity without having to construct $A$ but working with $P$ instead.\\
In order to construct $P$, in \cite{Riviere06} a beautiful yet a bit involved technique from Uhlenbeck's \cite{Uhlenbeck} is applied, which relies on a continuity argument and the implicit function theorem.\\
The purpose of this note is to remark the easy connection between the moving frame method H{\'{e}}lein developed in the 90's (\cite{Helein91}, see also \cite{Helein} and the appendix of \cite{Chone}) and Rivi\`{e}re's construction of the Coulomb gauge $P$. This implies a very easy proof for \cite[Lemma A.3]{Riviere06} which just consists of setting $P$ to be the minimizer of the following energy integral very well known from the moving frame technique
\begin{equation}\label{eq:energy}
 E(Q) := \int\limits_{B_1(0)} \abs{Q^T \nabla Q - Q^T \Omega Q}^2,\quad Q\in W^{1,2}(B_1(0),SO(n)).
\end{equation}
Here, $W^{1,2}(B_1(0),SO(n))$ are all those functions $Q \in W^{1,2}(B_1(0),\R^{n \times n})$ such that $Q(x)$ is an orthogonal matrix with $\det Q(x) > 0$ almost everywhere in $B_1(0)$. Neither is there any theory of Hardy and BMO spaces necessary, nor do we use an approximation of $\Omega$ or some kind of smallness conditions on $\Omega$, all of which is needed in the proof of \cite[Lemma A.3]{Riviere06}. Furthermore, all the estimates on $\nabla P$ as in \cite[Lemma A.3]{Riviere06} follow in a trivial way. Let us stress that as well smallness as also the duality of Hardy- and BMO-space is still needed in the proof of regularity later on, just not at this stage.\\
From this, one gets regularity of solutions to \eqref{eq:upde} just by applying a Dirichlet growth estimate for small exponents to
\begin{equation} \label{eq:trafopde}
 \operatorname{div} (P^T \nabla u) = (P^T\nabla P - P^T \Omega P) P^T \nabla u.
\end{equation}
The latter was done in \cite{StruweRiviere}. Although the Dirichlet growth approach cannot be applied without the fundamental fact that by \eqref{eq:divomegap} the quantity $(P^T\nabla P - P^T \Omega P) \nabla u$ lies in the Hardy space (cf. \cite{CLMS}), one can pinpoint the use of this information to exactly one inequality which can be proved in an elementary way bypassing Hardy-BMO theory (cf. \cite{Chanillo91}, \cite{ChanilloLi92}, \cite{HSZTBA}).\\
All in all, constructing $P$ by minimizing \eqref{eq:energy} as in \cite{Helein91}, and then using the Dirichlet growth theorem as in \cite{StruweRiviere} one gets a simplified proof of \cite[Theorem I.1]{Riviere06}. Interestingly, this simplification can be applied as well to the case of dimensions greater than two: In order to prove \cite[Theorem 1.1]{StruweRiviere} one does not need to prove that $P$ belongs to some Morrey-space. The $L^2$-estimates on the gradient of $P$ resulting from minimizing \eqref{eq:energy} are sufficient.\\
\\
As comparison, let us shortly remind the reader of some steps of the moving frame technique - for more details the reader is referred to \cite{Helein} as well as the appendix of \cite{Chone}: Let $v \in W^{1,2}(B_1(0),\mathcal{N})$ weakly satisfy
\begin{equation}\label{eq:lapvorth}
 \lap v\ \bot\ T_v\mathcal{N} \quad \mbox{in $B_1(0)$},
\end{equation}
where $\mathcal{N}$ is an $n$-dimensional compact manifold which is isometrically embedded in $\R^N$. Thus, orthogonality means orthogonality in the sense of the Euclidean metric in $\R^N$. Assume furthermore that there is some moving frame on $(\mathcal{N},T\mathcal{N})$: That is, there are smooth tangent vectors $e_i: \mathcal{N} \to T\mathcal{N}$, $1 \leq i \leq n$, such that at any point $y \in \mathcal{N}$ the $e_i(y)$ build an orthonormal basis of the tangential space $T_y\mathcal{N}$. It is then not too difficult to see, that by \eqref{eq:lapvorth}
\[
 \operatorname{div} (\langle e_i(v), \nabla v\rangle ) = \langle e_i(v), \nabla e_k(v) \rangle\ \langle e_k(v), \nabla v \rangle, \quad 1 \leq i \leq n.
\]
The scalar product $\langle \cdot, \cdot \rangle$ denotes the Euclidean scalar product in $\R^N$, that is $\langle e_i(v), \nabla v\rangle := \sum_{a=1}^N e^a_i(v) \nabla v^a$. Setting $\Omega_{ij} :=\langle e_i, \nabla e_j \rangle$ one observes the similarity with \eqref{eq:upde} - instead of $\nabla u^i$ in \eqref{eq:upde}, here we have $\langle e_i(v), \nabla v \rangle$. But from the point of view of growth estimates regarding $\nabla v$ this is not a big difference: Pointwise a.e. one can compare the size of $\left (\langle e_i(v), \nabla v \rangle \right )_{i=1}^n$ to the size of $\nabla v$.\\
The next step is to transform this moving frame $(e_i\circ v)_{i=1}^n$ into one that is more suitable for our equation, namely we seek $f_i = P^T_{ik}\ e_k\circ v$, where $P \in W^{1,2}(B_1(0),SO(n))$ is almost everywhere an orthogonal matrix in $\R^{n\times n}$, such that
\[
 0 = \operatorname{div} (\langle f_i, \nabla f_j \rangle) = \operatorname{div} (P^T_{ik} \nabla P_{jk} + P^T_{ik}\langle e_k(v), \nabla e_l(v) \rangle P_{jl}).
\]
Again, one should compare the latter expression to \eqref{eq:divomegap} with $\Omega_{ij}$ replaced by $\langle e_i(v), \nabla e_j(v) \rangle$. The point is, the moving frame technique and Rivi\`{e}re's approach in \cite{Riviere06} are very similar. The crucial additional ingredient in the latter is that one does not need to construct a moving frame $(e_i)_{i=1}^n$ in order to get an antisymmetric structure on the right hand side of certain Euler-Lagrange equations. In fact, this structure can be observed even in cases where one does not know how to get a moving frame like $(e_i)_{i=1}^n$ to start with.\\
\\
Let us stress that in the original regularity proof in \cite{Riviere06} which from the gauge transformation $P$ constructs the somewhat more elegant transformation $A$ satisfying \eqref{eq:divomegaa}, the main focus lies on the construction of good conservation laws for equations like \eqref{eq:upde}. That way one e.g. can avoid a Dirichlet Growth estimate below the natural exponent. Moreover, convergence issues become easier - once the preliminary work of constructing $P$ and then $A$ is done.\\
The connection between the techniques of minimizing the energy as in \eqref{eq:energy} and the construction of a Coulomb gauge by methods of Uhlenbeck is not new. In fact, in \cite{Wang05} in order to construct a moving frame for $n$-harmonic maps Uhlenbeck's approach is used. This is necessary because it is not clear how to obtain $W^{1,n}$-estimates of the transformation $P$ resulting from the $W^{1,2}$-minimization \eqref{eq:energy}.\\
\\
The structure of this note is as follows: In Section \ref{sec:energy} we will state the construction of $P$ to solve \eqref{eq:divomegap} by minimizing \eqref{eq:energy}. Section \ref{sec:nonash} contains a remark on how to avoid Nash-Moser's isometric imbedding theorem in order to prove Hildebrandt's conjecture. Finally, in the appendix we will sketch how to derive regularity from systems like \eqref{eq:trafopde} given that \eqref{eq:divomegap} is satisfied. There we also remark, that the $L^2$-estimates resulting from minimizing \eqref{eq:energy} are enough to prove partial regularity in dimensions $m > 2$ as in \cite{StruweRiviere}.\\
\\
As for our \textit{notation}, for a matrix or tensor $A$ we will denote $\abs{A}$ to be the Hilbert-Schmidt-norm of this quantity.\\
Mappings like the solution $u$ of $\eqref{eq:upde}$ will usually map the unit ball $B_1(0) \subset \R^m$ into the $n$-dimensional target manifold $\mathcal{N} \subset \R^N$ or simply into $\R^n$. Most of the time, instead of the Ball $B_1(0)$ one could use other kinds of sets to obtain the same results.\\
By  $\nabla = [\partial_1,\partial_2,\ldots,\partial_m]^T$ we denote the gradient. If $m =2$ the formally orthogonal gradient will be denoted by $\nabla^\bot = [-\partial_2,\partial_1]^T$.\\
The special orthogonal group in $\R^{n\times n}$ is denoted by $SO(n)$; $so(n)$ are all those matrices $(A_{ij})_{ij} \in \R^{n\times n}$ such that $A_{ij} = - A_{ji}$.\\
Many times, our constants depend on the dimensions involved. Further dependencies are usually clarified by a subscript. That is, a constant $C_p$ may depend on the dimensions as well as on $p$. Without further notice constants denoted by $C$ may change from line to line.

\vspace{2ex}\noindent
{\bf Acknowledgement.} It is a pleasure to thank Pawe{\l} Strzelecki for motivating the author to write this note down and for his and the University of Warsaw's hospitality.\\
\section{Direct Construction of Coulomb-Gauge}
\label{sec:energy}
In this section we prove, by elementary methods, the following theorem:
\begin{theorem}\label{th:energy}
(\cite{Helein91}, \cite[Lemma A.4, A.5]{Chone}; \cite[Chapter 4]{Helein}; \cite[Lemma~2.7]{Uhlenbeck}, \cite[Lemma~A.3]{Riviere06})\\
Let $D \subset \R^m$ be a smoothly bounded domain, $\Omega_{ij} \in L^2(D,\R^m)$, $\Omega_{ij} = -\Omega_{ji}$. Then there exists $P \in W^{1,2}(D,SO(n))$ such that
\[
 \operatorname{div} (P^T \nabla P - P^T \Omega P) = 0 \quad \mbox{in $D$,}
\]
and
\[
 \Vert \nabla P \Vert_{L^2(D)} + \Vert P^T \nabla P - P^T \Omega P \Vert_{L^2(D)} \leq 3\Vert \Omega \Vert_{L^2(D)}
\]
holds.
\end{theorem}%
There are mainly two approaches. A more general but involved method is due to Uhlenbeck in \cite[Lemma 2.7]{Uhlenbeck}; for the version needed here one best consults \cite[Lemma A.3]{Riviere06}. In \cite{MuellerS} this technique is also explained in some detail. The advantage of this version is that it works in similar ways in higher dimensions and for different integrability exponents. In \cite{MeyerRiviere03}, \cite{StruweRiviere} there is a Morrey-space version of it. The disadvantage is that it is technically involved, highly indirect - it is based on the implicit function theorem and a continuity argument - and needs already the theory of Hardy spaces in form of the duality between Hardy-space and BMO in order to derive the estimates on $\nabla P$.\\
The proof of Theorem \ref{th:energy} which we like to present here, follows from the next two lemmata which use only standard calculus of variation and a bit of Linear Algebra.
\begin{lemma}(cf. \cite{Chone}, Lemma A.4)\label{la:choneex}\\
Let $D \subset \R^m$ be a bounded domain. For any $\Omega_{ij} \in L^2(D,\R^m)$, $1 \leq i,j \leq n$, there exists $P \in W^{1,2}(D,SO(n))$ minimizing the variational functional
\[
 E(Q) = \intl_{D} \abs{Q^T \nabla Q - Q^T \Omega Q}^2, \quad Q \in W^{1,2}(D,SO(n)).
\]
Furthermore, $\Vert \nabla P \Vert_{L^2(D)} \leq 2\Vert \Omega \Vert_{L^2(D)}$.
\end{lemma}
\begin{remark}
Of course, this Lemma holds as well, if one takes 'Dirichlet'-boundary data, that is, if one assumes $Q-I \in W^{1,2}_0(D,\R^{n\times n})$, where $I$ is the $n$-dimensional identity matrix.
\end{remark}

\begin{lemma}(cf. \cite[Lemma A.5]{Chone})\label{la:chonediv}\\
Critical points $P \in W^{1,2}(D,SO(n))$ of
\[
 E(Q) = \intl_{D} \abs{Q^T \nabla Q - Q^T \Omega Q}^2, \quad Q \in W^{1,2}(D,SO(n)),
\]
satisfy
\[
 \operatorname{div} (P^T_{ik} \nabla P_{kj} - P^T_{ik} \Omega_{kl} P_{lj}) = 0, \quad 1\leq i,j \leq n,
\]
provided that $\Omega_{ij} \in L^2(D,\R^m)$ and $\Omega_{ij} = -\Omega_{ji}$ for any $1 \leq i,j \leq n$.\\
\end{lemma}%
\begin{proof}[Proof of Lemma \ref{la:choneex}]
The function $Q \equiv I := (\delta_{ij})_{ij}$ is clearly admissible. Thus, there exists a minimizing sequence $Q_k \in W^{1,2}(D,SO(n))$ such that
\[
 E(Q_k) \leq E(I) = \Vert \Omega \Vert_{L^2}^2, \quad k \in \N.
\]
By a.e. orthogonality of $Q_k(x) \in SO(n)$ we know that $Q_k(x)$ is bounded and
\[
 \abs{\nabla Q_k} = \abs{Q_k^T \nabla Q_k} \leq \abs{Q_k^T \nabla Q_k - Q_k^T \Omega Q_k} + \abs{\Omega}\quad \mbox{a.e. in $D$}; 
\]
thus
\[
 \Vert \nabla Q_k \Vert_{L^2(D)}^2 \leq 2(E(Q_k) + \Vert \Omega \Vert_{L^2(D)}^2) \leq 4\Vert \Omega \Vert_{L^2(D)}^2.
\]
Up to choosing a subsequence, we can assume that $Q_k$ converges weakly in $W^{1,2}$ to $P \in W^{1,2}(D,\R^{m\times m})$. At the same time it shall converge strongly in $L^2$, and pointwise almost everywhere. The latter implies $P^T P = \lim_{k\to \infty} Q_k^T Q_k = I$, and $\det(P) = 1$, that is $P \in SO(n)$ almost everywhere.\\
Denoting $\Omega^P := P^T \nabla P - P^T \Omega P$ we obtain
\[
 Q_k^T \nabla Q_k - Q_k^T \Omega Q_k = (P^T Q_k)^T \nabla (P^T Q_k) + (P^T Q_k)^T \Omega^P (P^T Q_k),
\]
and consequently
\[ 
\begin{ma}
 \abs{Q_k^T \nabla Q_k - Q_k^T \Omega Q_k}^2 &=& \abs{\nabla (P^T Q_k) + \Omega^P P^T Q_k}^2\\ 
&=& \abs{\nabla (P^T Q_k)}^2 + 2 \langle \nabla (P^T Q_k), \Omega^P P^T Q_k \rangle + \abs{\Omega^P}^2,
\end{ma}
\]
where in this case $\langle \cdot , \cdot\rangle$ is just the Hilbert-Schmidt scalar product for matrices. This implies
\[
\begin{ma}
 E(Q_k) &=& \intl_{D} \abs{\nabla (P^T Q_k)}^2 + 2 \langle \nabla (P^T Q_k), \Omega^P P^T Q_k \rangle + E(P)\\
&\geq& \intl_{D} \abs{\nabla (P^T Q_k)}^2 + 2 \intl_{D} \langle \nabla (P^T Q_k), \Omega^P P^T Q_k \rangle + \inf_{Q} E(Q).
\end{ma}
\]
The middle part of the right hand side converges to zero as $k \to \infty$. To see this, one can check that $\Omega^P P^T Q_k$ converges to $\Omega^P$ almost everywhere. Lebesgue's dominated convergence theorem implies strong convergence in $L^2$. On the other hand, $\nabla (P^T Q_k)$ converges to zero weakly in $L^2$.\\
Hence, using $E(Q_k) \xrightarrow{k \to \infty} \inf_{Q} E(Q)$, we have strong $W^{1,2}$-convergence of $P^T Q_k$ to $I$: Thus, $Q_k$ converges strongly to $P$, which readily implies minimality of $P$.\\
\end{proof}%
\begin{proof}[Proof of Lemma \ref{la:chonediv}]
Let $P$ be a critical point of $E(Q)$. A valid perturbation $P_\varepsilon$ is the following
\[
 P_\varepsilon := P e^{\varepsilon \varphi \alpha} = P + \varepsilon \varphi P\alpha + o(\varepsilon) \in W^{1,2}(D,SO(n))
\]
for any $\varphi \in C^\infty(\overline{D})$, $\alpha \in so(n)$ and $\varepsilon \to 0$. 
This uses the simple algebraic fact that the exponential function applied to a skew-symmetric matrix is an orthogonal matrix; or from the point of view of geometry, that the space of skew-symmetric matrices is the tangential space to the manifold $SO(n) \subset \R^{n\times n}$ at the identity matrix. Then,
\[
 P^{T}_\varepsilon = P^T - \varepsilon \varphi \alpha P^T + o(\varepsilon),
\]
\[
 \nabla P_\varepsilon = \nabla P + \varepsilon \varphi \nabla P\ \alpha + \varepsilon \nabla \varphi\ P \alpha + o(\varepsilon).
\]
Thus, denoting again $\Omega^P := P^T \nabla P - P^T \Omega P \in so(n)\otimes \R^m$, we obtain
\[
 \Omega^{P_\varepsilon} = \Omega^P + \varepsilon \varphi (\Omega^P \alpha - \alpha \Omega^P) + \varepsilon \nabla \varphi \alpha+ o(\varepsilon).
\]
The matrix $\Omega^P \alpha - \alpha \Omega^P$ is symmetric by antisymmetry of $\Omega^P$ and $\alpha$ which yields
\[
\sum_{i,j} (\Omega^P)_{ij}\cdot (\Omega^P \alpha - \alpha \Omega^P)_{ij} = 0 \mbox{\quad pointwise almost everywhere.}
\]
It follows that,
\[
 \abs{\Omega^{P_\varepsilon}}^2 = \abs{\Omega^P}^2 + 2 \varepsilon (\Omega^P)_{ij} \alpha_{ij} \nabla \varphi + o(\varepsilon),
\]
which readily implies
\[
 0 = \frac{d}{d\varepsilon} \bigg \vert_{\varepsilon = 0} E(P_\varepsilon) = \intl_{D} (\Omega^P)_{ij} \alpha_{ij}\cdot \nabla \varphi.
\]
This is true for any $\varphi \in C^\infty(\overline{D})$ and $\alpha \in so(n)$. Setting for arbitrary $1 \leq s,t \leq n$ our $\alpha_{ij} := \delta^s_i \delta^t_j - \delta^s_j \delta^t_i$, we arrive at
\[
 \operatorname{div} (\Omega^P)_{st} = 0 \quad \mbox{in $D$,}\quad 1 \leq s,t \leq n.
\]
\end{proof}
\begin{remark}
The disadvantage of this method is the fact that we do not know of a short and direct way to get better estimates on $P$ than the ones obtained here. That is, it does not seem to be clear that $\Omega \in L^p$ yields $P \in W^{1,p}$. On the other hand, this technique can be easily adapted to e.g. the case of different measures instead of the Lebesgue measure.\\
Interestingly, the knowledge that $\Vert \nabla P \Vert_{L^2} \leq C\ \Vert \Omega \Vert_{L^2}$ is sufficient also for partial regularity in dimensions $m > 2$. We will observe this in the appendix by a tiny modification of the proof in \cite{StruweRiviere}.
\end{remark}

\section{Hildebrandt's conjecture}
\label{sec:nonash}
In this section we sketch a proof of Hildebrandt's conjecture \cite{Hil82}, \cite{Hil83} stating that critical points of conformally invariant variational functionals on maps $v \in W^{1,2}(D,\R^n)$ where $D \subset \R^2$ are continuous: We construct from Gr\"uter's \cite{Grueter} characterization directly a Rivi\`{e}re-type system - avoiding the Nash-Moser-embedding theorem as in e.g. \cite{Chone} and \cite[Theorem I.2]{Riviere06}.\\
As explained for example in \cite[Section 1.2]{Helein}, the Nash-Moser-theorem is used to avoid the appearance of terms involving Christoffel-symbols in the Euler-Lagrange equations of harmonic maps or - more generally - conformally invariant variational functionals:
Let $D \subset \R^2$ be an open set. For $v \in W^{1,2}(D,\R^n)$ we define the functional
\[
 \mathcal{F}(v) \equiv \mathcal{F}_D(v) = \intl_D F(v(x),\nabla v(x))\ dx,
\]
where $F: \R^n \times \R^{2n} \to \R$ is of class $C^1$ with respect to the first entry and of class $C^2$ with respect to the second entry. The functional $\mathcal{F}$ is called conformally invariant if
\[
 \mathcal{F}_D(v) = \mathcal{F}_{D'}(v \circ \phi)
\]
for every smooth $v: D \to \R^n$ and every smooth conformal diffeomorphism $\phi: D' \to D$. Suppose $\mathcal{F}$ is conformally invariant and that for some $\Lambda > 0$
\[
 \frac{1}{\Lambda} \abs{p}^2 \leq F(v,p) \leq \Lambda \abs{p}^2 \quad \mbox{for all $v \in \R^n$, $p \in \R^{2n} $.}
\]
Then, by \cite[Theorem 1]{Grueter}, there exists a positive, symmetric matrix $(g_{ij})$ and a skew symmetric matrix $(b_{ij})$ such that
\[
 F(v,p) = g_{ij}(v) p^i\cdot p^k + b_{ij}(v) \det(p^i,p^j),
\]
and hence
\[
 \mathcal{F}(v) = \intl_D g_{ij}(v) \nabla v^i\cdot \nabla v^k + b_{ij}(v) \nabla v^i \cdot \nabla^\bot v^j.
\]
Recall that $\nabla^\bot = (-\partial_y,\partial_x)^\bot$. Let us interpret $(g_{ij})_{i,j=1}^n$ as a metric of the target space $\R^n$. As in \cite[(2.7)]{Grueter} Euler-Lagrange-equation could then be written as
\begin{equation} \label{eq:nnm:elwonm}
 2\lap v^i + \Gamma_{kl}^i(u) \nabla u^k \cdot \nabla u^l = g^{ij} \{ \partial_l b_{jk} + \partial_j b_{kl} + \partial_k b_{lj}\} (u)\ \nabla u^k \cdot \nabla^\bot u^l,
\end{equation}
where
\[
 \Gamma_{kl}^i = g^{ij} \{\partial_l g_{jk} - \partial_j g_{kl} + \partial_k g_{lm} \}
\]
are the Christoffel symbols corresponding to the metric $(g_{ij})$. Here, we have denoted the inverse of $(g_{ij})$ by $(g^{ij})$. Let
\[
 \Omega_{jk} := \{ \partial_l b_{jk} + \partial_j b_{kl} + \partial_k b_{lj}\} (u)\ \nabla^\bot u^l
\]
which is antisymmetric. Equation \eqref{eq:nnm:elwonm} then reads as
\begin{equation} \label{eq:nnm:elwonm2}
 2\lap u^i + \Gamma_{kl}^i(u) \nabla u^k \cdot \nabla u^l = g^{ij}(u)\ \Omega_{jk} \cdot \nabla u^k.
\end{equation}
At first glance, \eqref{eq:nnm:elwonm2} does not seem to fit into the setting of \eqref{eq:upde} because in general $(g_{ij})$ is not the standard Euclidean metric on $\R^n$.\\
The Nash-Moser-Theorem (cf. \cite{Nash56}, \cite{Kuiper55}, \cite{Guenther91}, \cite{Hamilton82}) solves this problem: It states that there is a manifold $\mathcal{N} \subset \R^{N}$, $N \geq n$, and a $C^1$-diffeomorphism $T$ mapping $(\R^n,g_{ij})$ isometrically into $(\mathcal{N},c_{ij})$ where $c_{ij}$ is the induced $\R^N$-metric on $\mathcal{N}$. That is, $T: (\R^n,g_{ij}) \to \mathcal{N}$ and
\begin{equation}\label{eq:nonash:isometry}
 \langle dT_x \left (\frac{\partial}{\partial x^i} \right ),dT_x \left (\frac{\partial}{\partial x^j} \right ) \rangle_{\R^N} = g_{ij}(x), \quad x \in \R^n, \quad 1 \leq i,j \leq n.
\end{equation}
Here, $\left ( \frac{\partial}{\partial x^i} \right )_{i=1}^n$ denotes the standard euclidean basis in $\R^n$. Using this isometric diffeomorphism $T$, we introduce an adapted functional $\widetilde {\mathcal{F}}$ defined on mappings $\tilde{v} \in W^{1,2}(D,\mathcal{N})$ of which $T(u)$ is a critical point. Looking at the Euler-Lagrange equations of this new $\widetilde{\mathcal{F}}$, the fact that the metric on $\mathcal{N}$ is induced by the surrounding space $\R^N$ will imply trivial Christoffel-symbols. On the other hand, the additional side-condition $\tilde{v}(x) \in \mathcal{N}$ a.e. will bring up a term involving the second fundamental form of the embedding $\mathcal{N} \subset \R^N$. This new term can be rewritten into the form of the right hand side of \eqref{eq:upde} as was observed in \cite{Riviere06}.\\
In fact, setting 
\[
\tilde{b}_{ab} := (dT^a\left (\frac{\partial}{\partial x^k} \right )\ g^{ki}\ b_{ij}\ g^{jl}\ dT^b \left (\frac{\partial}{\partial x^l} \right ))\circ T^{-1}
\]
 we obtain
\[
 \mathcal{F}(v) = \intl_D \abs{\nabla T(v)}^2_{\R^N} + \sum_{a,b = 1}^N\intl_D \tilde{b}_{ab}(Tv) \nabla T^a(v) \cdot \nabla^\bot T^b(v).
\]
Consequently, $u$ is a critical point of $\mathcal{F}$ if and only if $T(u)$ is a critical point of
\[
 \widetilde{\mathcal{F}}(\tilde{v}) = \intl_D \abs{\nabla \tilde{v}}^2 + \sum_{a,b = 1}^N\tilde{b}_{ab}(\tilde{v}) \nabla \tilde{v}^a \cdot \nabla^\bot \tilde{v}^b, \quad \tilde{v} \in W^{1,2}(D,\mathcal{N}).
\]
One checks that $\tilde{b}$ is antisymmetric. Hence, assuming that the second fundamental form of the embedding $\mathcal{N} \subset \R^N$ is bounded, one can proceed as in \cite[Theorem I.2]{Riviere06} to see that the Euler-Lagrange equation of $\widetilde{\mathcal{F}}$ is a system of type \eqref{eq:upde}. Thus, regularity of $T(u)$, $u$ is implied.\\
\\
The proof of the Nash-Moser embedding is quite involved. However, it can be avoided easily by the following approach: A critical point $u \in W^{1,2}(D,\R^n)$ of $\mathcal{F}$ weakly satisfies \eqref{eq:nnm:elwonm} or equivalently for $1 \leq j \leq n$
\begin{equation}\label{eq:nonash:el}
\begin{split}
 &-\operatorname{div} (2g_{jk}(u) \nabla u^k) + (\partial_j  g_{kl})(u) \nabla u^k \cdot \nabla u^l\\
= &\operatorname{div} (2b_{jk}(u) \nabla^\bot u^k) - (\partial_j  b_{kl})(u) \nabla u^k \cdot \nabla^\bot u^l.
\end{split}
\end{equation}

By algebraic calculations one constructs vector functions $e_i: \R^n \to \R^n$, $1 \leq i,j \leq n$, such that pointwise
\begin{equation}
\label{eq:nonash:eisom}
 \langle e_i, e_j \rangle_{\R^n} = g_{ij}.
\end{equation}
In order to construct $T$ as in \eqref{eq:nonash:isometry} one would be tempted to integrate, that is, to set
\[
 d T \left (\frac{\partial}{\partial x^i} \right ) := e_i,
\]
and therefore one would need $e_i$ satisfying \eqref{eq:nonash:eisom} and
\begin{equation}\label{eq:nonash:curle}
 \partial_j e_i - \partial_i e_j = 0,\quad 1\leq i,j\leq n.
\end{equation}
One observes now that the latter quantity is a skew symmetric one. That is, the error one would make in \eqref{eq:nonash:el} assuming \eqref{eq:nonash:curle} to hold is not a bad one - it fits into the setting of Rivi\`{e}re's system \eqref{eq:upde}. In fact, the following lemma holds, which by the techniques of \cite{StruweRiviere}, see also the appendix, Remark \ref{rm:dg:simsys}, implies regularity.
\begin{lemma}
Let $u \in W^{1,2}(D,\R^n)$ be a weak solution of
\begin{equation} \label{eq:nonash:elg} 
 -\operatorname{div}(2g_{ik}(u) \nabla u^k) + (\partial_i  g_{kl})(u) \nabla u^k \cdot \nabla u^l = \Omega_{ik} \cdot \nabla u^k + \nabla^\bot b_{ik} \nabla u^k.
\end{equation}
Assume that $g,g^{-1} \in W^{1,\infty}(\R^n,GL(n))$ are symmetric and positive definite, $b_{jk} \in W^{1,2}(D)$, $\Omega_{ij} = - \Omega_{ji} \in L^2(D,\R^2)$.\\
Then there are $A \in W^{1,2}\cap L^\infty(D,GL(m))$, $\widetilde{\Omega}_{ij} = - \widetilde{\Omega}_{ji} \in L^2(D,\R^2)$ such that
\[
 \operatorname{div} (A_{ik} \nabla u^k) = \widetilde{\Omega}_{ik}\cdot A_{kl} \nabla u^l + \nabla^\bot b_{ik} \cdot \nabla u^k.
\]
\end{lemma}
\begin{proof}[Sketch of the proof]
By easy algebraic transformations using symmetry and positive definiteness of $g$ one can choose $e_i \in W^{1,\infty}(\R^n,\R^n)$ such that 
\begin{equation}\label{eq:nonash:eij}
\langle e_i(x), e_j(x) \rangle_n = g_{ij}(x), \quad x \in \R^n,\quad 1\leq i,j \leq n. 
\end{equation}
The $A_{ia}$ from the claim will be $e_i^a\circ u$. Let us abbreviate as follows
\begin{equation}
\label{eq:nonash:xia}
 \xi^a := A_{ak} \nabla u^k = e_k^a(u) \nabla u^k, 
\end{equation}
which is equivalent to
\begin{equation}
\label{eq:nonash:uxia}
 \nabla u^j = g^{jk}(u)\ e_k^a(u)\ \xi^a.
\end{equation}
Let $\varphi$ be any admissible testfunction. The first term on the lefthand side of \eqref{eq:nonash:elg}
\[
 I := 2 g_{ik}(u) \nabla u^k\cdot \nabla \varphi^i \overset{\eqref{eq:nonash:uxia}}{=} 2 \xi^a \cdot (e^a_i(u) \nabla \varphi^i).
\]
On the other hand, the second term of \eqref{eq:nonash:elg}
\[
 \begin{ma}
 II &:=& \partial_{i} g_{kl}(u)\ \nabla u^k\cdot \nabla u^l\ \varphi^i\\
&\overset{\eqref{eq:nonash:eij}}{=}& 2 (\partial_{i} e_k^a)(u)\ e_l^a(u)\ \nabla u^k\cdot \nabla u^l\ \varphi^i\\
&=& 2 (\partial_{k} e_i^a)(u)\ e_l^a(u)\ \nabla u^k\cdot \nabla u^l\ \varphi^i\\
&&\quad +2 (\partial_{i} e_k^a - \partial_{k} e_i^a)(u)\ e_l^a(u)\ \nabla u^k\cdot \nabla u^l\ \varphi^i\\
&=:& II_1 + II_2.
 \end{ma}
\]
One computes
\[
 II_1 \overset{\eqref{eq:nonash:xia}}{=} 2\nabla (e_i^a(u))\ \varphi^i \cdot \xi^a,
\]
and thus
\[
 I + II_1 = 2\xi^a  \cdot \nabla (e_i^a(u) \varphi^i).
\]
For arbitrary $\tilde{\varphi} \in C_0^\infty(D,\R^n)$ one sets
\begin{equation}\label{eq:nonash:testfct}
 \varphi^i := g^{ij}(u)\ \langle e_j(u), \tilde{\varphi}\rangle_n
\end{equation}
which is an admissible testfunction. One checks that
\[
 \langle \tilde{\varphi} - e_j(u)\ \varphi^j, e_s(u)\rangle_n \overset{\eqref{eq:nonash:eij}}{=} 0, \quad 1 \leq s \leq n.
\]
Pointwise in $\R^n$ the vectors $e_i \in \R^n$, $1 \leq i \leq n$, are linearly independent, which implies $\tilde{\varphi} = e_j(u)\ \varphi^j$ almost everywhere. Then
\[
 I+II_1 = 2\xi^a \cdot \nabla \tilde{\varphi}^a.
\]
Rewriting the quantity $II_2$ in terms of $\xi^a$ and $\tilde{\varphi}$ yields
\[
\begin{ma} 
II_2 &=&  2(\partial_{i} e_k^a - \partial_{k} e_i^a)(u)\ \xi^a \cdot g^{ks}(u)\ e_s^b(u)\ \xi^b\ g^{it}(u)\ e_t^c(u)\ \tilde{\varphi}^c\\
&=:&2\omega_{bc}\ \xi^b\ \tilde{\varphi}^c,
\end{ma}
\]
where $\omega_{bc} = (\partial_{i} e_k^a - \partial_{k} e_i^a)(u)\ \xi^a \cdot g^{ks}(u)\ e_s^b(u)\ g^{it}(u)\ e_t^c(u)$ is antisymmetric and in $L^2$.\\
For the right hand side of \eqref{eq:nonash:elg} one observes just by plugging in \eqref{eq:nonash:testfct} and \eqref{eq:nonash:uxia}
\[
 \Omega_{ik}\cdot \nabla u^k \varphi^i = \Omega_{ik}\ g^{kl}(u)\ e_l^a(u)\ \ g^{is}(u) e^c_s(u)\ \ \tilde{\varphi}^c\cdot \xi^a
\]
and $\widetilde{\Omega}_{ac} :=\Omega_{ik}\ g^{kl}(u)\ e_l^a(u)\ \ g^{is}(u) e^c_s(u)$ is antisymmetric and in $L^2$.
\end{proof}
\renewcommand{\thesection}{A}
\renewcommand{\thesubsection}{A.\arabic{subsection}}
\section{Appendix: Application of Dirichlet Growth Theorem}\label{sec:dgt}
In this section we will sketch how to apply the Dirichlet Growth Theorem (cf. \cite[Theorem 3.5.2]{Morrey}) in order to derive regularity for solutions of \eqref{eq:upde}, given the existence of $P$ as in the proof of Theorem \ref{th:energy}. A detailed proof can be found in \cite{StruweRiviere}. As a slight modification, we will remark on how to avoid Morrey-space estimates on the gradient of the gauge-transformation $P$. Those Morrey-space estimates can be obtained via the Uhlenbeck-Approach, but it is not obvious how to get them by a method as in Theorem \ref{th:energy}. We will show that the $L^2$-estimates of Theorem \ref{th:energy} are sufficient.\\
\\
We will use one non-elementary technique, namely the duality between Hardy-space and BMO. But in fact we need only a special case. For $p \in (1,\infty)$ set
\[
 \mathcal{J}_p(x,\rho;f) := \frac{1}{\rho^{m-p}}\ \intl_{B_\rho(x)} \abs{f}^p,
\]
\[
 \mathcal{M}_p(y,\varrho;f) := \sup_{B_{\rho}(x) \subset B_{\varrho}(y)} \mathcal{J}_p(x,\rho;f).
\]
\begin{lemma}[Hardy-BMO-Inequality]\label{la:hbmo}
For any $p > 1$, there is a uniform constant $C_{m,p}$ such that the following holds:\\
For any ball $B \equiv B_\varrho (y) \subset \R^m$, $2B = B_{2\varrho}(y)$ the ball with same center and twice the radius, $a \in W^{1,2}(2B)$, $\Gamma \in L^2(B,\R^m)$, $\operatorname{div} \Gamma = 0$ in $B$, $c \in W^{1,2}_0\cap L^\infty(B)$
\[
\abs{ \intl_{B} (\nabla a \cdot \Gamma)\ c} \leq C_{m,p}\ \Vert \Gamma \Vert_{L^2(B)}\ \Vert \nabla c \Vert_{L^2(B)}\ \left (\mathcal{M}_p(y,2\varrho,\nabla a) \right )^{\frac{1}{p}},
\]
whenever the right hand side is finite.
\end{lemma}
For a proof one can use Hardy-space theory, (cf. \cite{CLMS}[Theorem~II.1], \cite[Chapter II.2]{FS72}, \cite[Chapter IV, \textsection 1.2]{Stein93}), but in this special case the proof is easier (cf. \cite{Chanillo91}, \cite{ChanilloLi92}, \cite{HSZTBA}).\\
\\
\begin{theorem}[{\cite[Theorem 1.1]{StruweRiviere}}]\label{th:srpartial}
There is $\varepsilon \equiv \varepsilon(m) \in (0,1)$ such that the following holds:\\
Let $D \subset \R^m$ be open and $u \in W^{1,2}(D,\R^n)$ be a solution of
\[
 \lap u^i = \Omega_{ik} \cdot \nabla u^k\quad \mbox{in $D$},\quad 1 \leq i \leq n
\]
such that
\begin{equation}\label{eq:sr:osmall}
 \sup_{B_r(x) \subset D} \frac{1}{r^{m-2}} \intl_{B_r(x)} \abs{\Omega}^2 \leq \varepsilon
\end{equation}
and
\begin{equation}\label{eq:sr:ubounded}
 \sup_{B_r(x) \subset D} \frac{1}{r^{m-2}} \intl_{B_r(x)} \abs{\nabla u}^2 < \infty.
\end{equation}
If $\Omega_{ij} = - \Omega_{ji} \in L^2(D,\R^m)$ then $u \in C^{0,\alpha}(D,\R^n)$ for some $\alpha \in (0,1)$.
\end{theorem}
\begin{proof}[Sketch of the proof]
Most parts of the following are a copy of the proof in \cite[Theorem 1.1]{StruweRiviere}.\\
Let $z \in D$, $0 < r < R < \frac{1}{2}\operatorname{dist}(z,\partial D)$. Apply Theorem \ref{th:energy} on $B_R(z)$: There exists $P \in W^{1,2}(B_R(z),SO(n))$ such that
\begin{equation}\label{eq:sr:divope0}
 \operatorname{div} (\Omega^P) \equiv \operatorname{div} (P^T \nabla P - P^T \Omega P) = 0 \quad \mbox{weakly in $B_R(z)$},
\end{equation}
with the estimate
\begin{equation}\label{eq:sr:npopest}
 \Vert \nabla P \Vert_{L^2(B_R(z))} + \Vert \Omega^P \Vert_{L^2(B_R(z))} \leq 3 \Vert \Omega \Vert_{L^2(B_R(z))}.
\end{equation}
We have weakly
\begin{equation}\label{eq:sr:divpnu}
 \operatorname{div} (P^T \nabla u) = \Omega^P\cdot  P^T \nabla u \quad \mbox{in $B_R(z)$}.
\end{equation}
Use Hodge decomposition to find $f \in W^{1,2}_0(B_R(z), \R^n)$, $g \in W^{1,2}_0(B_R(z), \wedge^2 \R^n)$, $h \in C^\infty(B_R(z),\R^n \otimes \R^m)$ such that
\begin{equation}\label{eq:sr:hdec}
 P^T \nabla u = \nabla f + \operatorname{Curl} g + h \quad \mbox{a.e. in $B_R(z)$,}
\end{equation}
\begin{equation}\label{eq:sr:lapf}
 \begin{cases}
  \lap f = \operatorname{div} (P^T \nabla u) \overset{\eqref{eq:sr:divpnu}}{=}\Omega^P\cdot  P^T \nabla u \quad &\mbox{in $B_R(z)$},\\
  f = 0 \quad &\mbox{on $\partial B_R(z)$},
 \end{cases}
\end{equation}
\[
 \begin{cases}
  \lap g = \operatorname{curl} (P^T \nabla u) \quad &\mbox{in $B_R(z)$},\\
  g = 0 \quad &\mbox{on $\partial B_R(z)$},
 \end{cases}
\]
\[
 \begin{cases}
  \operatorname{div} h = 0   \quad &\mbox{in $B_R(z)$},\\
\operatorname{curl} h = 0   \quad &\mbox{in $B_R(z)$}.\\
 \end{cases}
\]
For more on Hodge-decompositions we refer to \cite[Corollary 10.5.1]{IwaniecMartin}. Fix $1 < p < \frac{m}{m-1}$. One estimates
\[
\begin{ma}
 \intl_{B_r(z)} \abs{\nabla u}^p &=& \intl_{B_r(z)} \abs{P^T \nabla u}^p\\
&\overset{\eqref{eq:sr:hdec}}{\leq}& C_p \left ( 
	\intl_{B_r(z)} \abs{h}^p + \intl_{B_R(z)} \abs{\nabla f}^p  + \intl_{B_R(z)} \abs{\nabla g}^p
\right ).
\end{ma}
\]
By harmonicity we have (cf. \cite[Theorem 2.1, p.78]{GiaquintaMultipleIntegrals})
\[
 \intl_{B_r(z)} \abs{h}^p \leq C_p \left (\frac{r}{R}\right )^m \intl_{B_R(z)} \abs{h}^p.
\]
Consequently, again by \eqref{eq:sr:hdec},
\begin{equation} \label{eq:sr:estnu}
 \begin{ma}
 \intl_{B_r(z)} \abs{\nabla u}^p &\leq& C_p \left ( 
	\left (\frac{r}{R}\right )^m \intl_{B_R(z)} \abs{\nabla u}^p + \intl_{B_R(z)} \abs{\nabla f}^p  + \abs{\nabla g}^p
\right ).
\end{ma}
\end{equation}
In order to estimate $\int_{B_R(z)} \abs{\nabla f}^p$ note that since $f = 0$ on $\partial B_R(z)$, by duality
\begin{equation}\label{eq:sr:rrt}
 \Vert \nabla f \Vert_{L^p(B_R(z))} \leq C_p \sup_{
\ontop{\varphi \in C_0^\infty(B_R(z))}{ \Vert \varphi \Vert_{W^{1,q}} \leq 1}
} \intl_{B_R(z)} \nabla f\cdot \nabla \varphi. 
\end{equation}%
Here, $q = \frac{p}{p-1}$ denotes the H\"older-conjugate exponent of $p$. If $\Vert \varphi \Vert_{W^{1,q}(B_R(z))} \leq 1$ one calculates
\begin{equation}\label{eq:sr:vpest}
 \Vert \varphi \Vert_{L^\infty(B_R(z))} \leq C_p\ R^{1+\frac{m}{p} -m},\quad  \Vert \nabla \varphi \Vert_{L^2(B_R(z))} \leq C_p\  R^{\frac{m}{p}-\frac{m}{2}}.
\end{equation}%
Note that the $L^\infty$-bound holds only as $q > m$ by choice of $p$. In particular, the constant $C_p$ blows up as $p$ approaches $\frac{m}{m-1}$ from below.\\
Recall our notation
\[
 \mathcal{J}_p(x,\rho) := \frac{1}{\rho^{m-p}}\ \intl_{B_\rho(x)} \abs{\nabla u}^p,
\]
\[
 \mathcal{M}_p(y,\varrho) := \sup_{B_{\rho}(x) \subset B_{\varrho}(y)} \mathcal{J}_p(x,\rho).
\]
By \eqref{eq:sr:lapf}, 
\[
\begin{ma} 
\intl_{B_R(z)} \nabla f \cdot \nabla \varphi &=& \intl_{B_R(z)} \Omega^P\cdot P^T \nabla u\ \varphi.
\end{ma}
\]
As of \eqref{eq:sr:divope0} Lemma \ref{la:hbmo} can be applied to this quantity by choosing $c = P^T_{kl} \varphi$, $a = u^l$, $\Gamma = (\Omega^P)_{ik}$ for any $1 \leq i,k,l \leq n$. Then \eqref{eq:sr:rrt} is further estimated by
\[
\begin{ma}
&& \Vert \nabla f \Vert_{L^p(B_R(z))}\\ 
&\leq& C_p\ \Vert \Omega^P \Vert_{L^2(B_R(z))}\ (\Vert \nabla P \Vert_{L^2(B_R(z))} \Vert \varphi \Vert_{L^\infty} + \Vert \nabla \varphi \Vert_{L^2})\ \left (\mathcal{M}_p (z,2R) \right )^{\frac{1}{p}}\\ 
&\overset{\eqref{eq:sr:npopest}}{\leq}& C_p\ \Vert \Omega \Vert_{L^2(B_R(z))}\ (\Vert \Omega \Vert_{L^2(B_R(z))}\ \Vert \varphi \Vert_{L^\infty}   + \Vert \nabla \varphi \Vert_{L^2})\ \left (\mathcal{M}_p (z,2R) \right )^{\frac{1}{p}}\\
&\overset{\ontop{\eqref{eq:sr:osmall} }{ \eqref{eq:sr:vpest} }}{\leq}& C_p\ \varepsilon\ R^{\frac{m}{p} -1}\ \left (\mathcal{M}_p (z,2R) \right )^{\frac{1}{p}}.
\end{ma}
\]
Note again that the constant $C_p$ blows up as $p$ approaches $\frac{m}{m-1}$ from below. The last step is the only qualitative albeit tiny difference to the proof in \cite{StruweRiviere}: Instead of using an a-priori estimate on $\sup_{r} \frac{1}{r^{m-2}} \intl_{B_r} \abs{\nabla P}^2$ and $\sup_{r} \frac{1}{r^{m-2}} \intl_{B_r} \abs{\Omega^P}^2$, we use the domain-independent estimate \eqref{eq:sr:npopest} of the $L^2$-Norm of $\nabla P$ and $\Omega^P$, respectively.
By a similar argument
\[
 \Vert \nabla g \Vert_{L^p(B_R(z))} \leq C_p\ \varepsilon\ R^{\frac{m}{p} -1}\ \left (\mathcal{M}_p (z,2R) \right )^{\frac{1}{p}}.
\]
Plugging these estimates into \eqref{eq:sr:estnu} we arrive at
\[
 \intl_{B_r(z)} \abs{\nabla u}^p \leq C_p\ \left (\frac{r}{R}\right )^{m} \intl_{B_R(z)} \abs{\nabla u}^p + C_p\ \varepsilon\ R^{m-p}\ \mathcal{M}_p (z,2R).
\]
The right hand side of this estimate is finite by \eqref{eq:sr:ubounded}. We divide by $r^{m-p}$ to get
\[
\frac{1}{r^{m-p}} \intl_{B_r(z)} \abs{\nabla u}^2
\]
\[
\leq C_p\ \left (\frac{r}{R}\right )^{p} \frac{1}{R^{m-p}}\intl_{B_R(z)} \abs{\nabla u}^p + C_p\ \varepsilon\ \left (\frac{R}{r}\right )^{m-p}\ \mathcal{M}_p (z,2R).
\]
Hence,
\[
 \mathcal{J}_p(z,r) \leq C_p \left ( \left (\frac{r}{R}\right )^{p} + \varepsilon\ \left (\frac{R}{r}\right )^{m-p} \right )\ \mathcal{M}_p(z,2R).
\]
Choose $\gamma \in (0,\frac{1}{2})$ such that $C_p \gamma^p \leq \frac{1}{4}$ and set $\varepsilon := \gamma^m$. Then for $r := \gamma R$ we have shown
\[
 \mathcal{J}_p(z,\gamma R) \leq \frac{1}{2}\ \mathcal{M}_p(z,2R).
\]
This is valid for any $R > 0$, $z \in D$ such that $B_{2R}(z) \subset D$. For arbitrary $\rho \in (0,1)$, $y \in D$, $B_{2\rho}(y) \subset D$ this implies
\[
 \mathcal{J}_p(z,\gamma R) \leq \frac{1}{2}\ \mathcal{M}_p(y,\rho) \quad \mbox{whenever $B_{2R}(z) \subset B_\rho(y)$,}
\]
that is
\[
 \mathcal{M}_p(y,\frac{\gamma}{2} \rho) \leq \frac{1}{2} \mathcal{M}_p(y,\rho). 
\]
This gives H\"older-continuity as claimed.\\
\end{proof}

\begin{remark}
With the presented techniques one can prove slight generalizations of this. For example, in order to prove regularity for systems of the type
\[
 \partial_\alpha (g_{\alpha \beta} \partial_\beta u^i) = g_{\alpha \beta}\ \Omega_{ik}^\beta\ \nabla u^k,
\]
one would minimize
\[
 E(P) = \intl_D (P^T_{ik} \partial_\alpha P_{kj} - P_{ik}^T \Omega_{kl}^\alpha P_{lj})\ g_{\alpha \beta}\ (P^T_{ik} \partial_\beta P_{kj} - P_{ik}^T \Omega_{kl}^\beta P_{lj}).
\]
\end{remark}

\begin{remark}\label{rm:dg:simsys}
Slightly modifying this approach, one also can check the following: Let $\xi^i := A_{ik} \nabla u^k$, $A \in W^{1,2}\cap L^\infty(D,\R^n)$, and $u \in W^{1,2}(D,\R^m)$ satisfy \eqref{eq:sr:ubounded}. Assume that $\xi$ is a solution of a system like
\[
 \operatorname{div}(\xi^i) = \Omega_{ik} \cdot \xi^{k}\quad \mbox{in $D$},\quad 1\leq i \leq n.
\]
This implies better regularity of $u$, if \eqref{eq:sr:osmall} holds for $\Omega$ and $A$ and under the additional condition that there is a uniform constant $\Lambda > 0$ such that
\[
 \frac{1}{\Lambda} \abs{\xi} \leq \abs{\nabla u} \leq \Lambda \abs{\xi} \quad \mbox{a.e. in $D$}.
\]
The last condition is used to switch in growth estimates like \eqref{eq:sr:estnu} between $\abs{\xi}$ and $\abs{\nabla u}$.
\end{remark}

\bibliographystyle{alpha}%
\bibliography{bib}%
\vspace{2em}
\begin{tabbing}
\quad\=Armin Schikorra\\
\>RWTH Aachen University\\
\>Institut f\"ur Mathematik\\
\>Templergraben 55\\
\>52062 Aachen\\
\>Germany\\
\\
\>email: schikorra@instmath.rwth-aachen.de
\end{tabbing}
\end{document}